\documentclass{article}
\usepackage{amsmath}
\usepackage{amsthm}
\usepackage{indentfirst}
\usepackage{subfigure}
\usepackage{graphicx}
\usepackage{epsfig}
\usepackage[body={16cm,24cm}, top=3cm]{geometry}
\newtheorem{theorem}{Theorem}
\newtheorem{proposition}{Proposition}
\newtheorem{corollary}{Corollary}

\newtheorem{lemma}{Lemma}
\newtheorem{remark}{Remark}
\usepackage{amssymb}
\usepackage{indentfirst}
\author{Wang Changlong, Zhou Feng}
\title{Exact sparse reconstruction form Vandermonde matrices}
\begin{document}\large
\title{Exact sparse reconstruction form Vandermonde matrices}
\date{}
\maketitle
\begin{abstract}
As a conclusion in classical linear algebra, an underdetermined linear equations usually have an infinite number of solutions. The sparest one among these solutions is significant in many applications. This problem can be modeled as the following $l_0$-minimization,
\begin{eqnarray}
\notag\mathop{\min}\limits_{x \in \mathbb{R}^m} \|x\|_0 \ s.t. \ Ax=b.
\end{eqnarray}
However, to find the sparsest solution of an underdetermined linear equations is NP-hard. Therefore, an important approach to solve the following $l_p$-minimization ($0<p\leq1$),
\begin{eqnarray}
\notag\mathop{\min}\limits_{x \in \mathbb{R}^m} \|x\|_p^p \ s.t. \ Ax=b,
\end{eqnarray}
The purpose of this problem is to find a $p$-norm minimization solution $(0<p\leq1)$ instead of the sparest one.

In this paper, we study the equivalence relationship between $l_0$-minimization and $l_p$-minimization with a high sparse level. Compared with most of related work which adopt Restricted Isometry Property (RIP) and Restricted Isometry Constant (RIC), these result only solve the situation when the solution $\breve{x}$ of $l_0$-minimization satisfies that $\|\breve{x}\|_0<spark(A)/2$. Therefore, the main contribution of this paper is to give an analytic expression $p^*$ such that $l_p$-minimization is equivalent to $l_0$-minimization when $\|\breve{x}\|_0<\frac{spark(A)}{2}$, and we also prove this result is true even $\|\breve{x}\|_0>\frac{spark(A)}{2}$ for Vandermonde matrices.

Compared with the similar results based on RIP and RIC, the result of this paper do not need the uniqueness assumption, i.e., the solution $x^*$ of $l_0$-minimization do not have to be assumed to be the unique solution which is the main breakthrough in our result. Another superiority of our result is its computability, i.e., each part in the analytic expression can be easily calculated.
\end{abstract}
\textbf{keywords:} sparse recovery, Vandermonde matrix, spark, $l_p$-minimization
\section{Introduction}
As a conclusion in classical linear algebra, the underdetermind linear equations $Ax=b$ usually admit an infinite number of solutions. To find the sparsest one in these solutions is actually the is the key issue in many applications such as visual coding \cite{olshausen1996emergence}, matrix completion \cite{candes2009exact}, source localization \cite{malioutov2005sparse}, and face recognition \cite{wright2009robust}, all these problems are popularly modeled into the following $l_0$-minimization:
\begin{eqnarray}
\mathop{\min}\limits_{x \in \mathbb{R}^n} \|x\|_0\ s.t. \ Ax=b
\end{eqnarray}

where $A \in \mathbb{R}^{m \times n}$ is an underdetermined matrix (i.e. $m<n$), and $\|x\|_0$ indicates the number of nonzero elements of $x$, which is commonly called $l_0$-norm although it is not a real vector norm.

However, Natarajan \cite{natarajan1995sparse} proved that to find the sparest solution of an underdetermind linear equations is NP-hard and $l_0$-minimization is also combinational and computationally intractable because of the discrete and discontinuous nature. Therefore, alternative strategies to find sparest solution have been put forward (see, for example \cite{candes2008restricted,candes2006near,foucart2009sparsest,donoho2005sparse,tropp2004greed,petukhov2006fast,candes2005decoding}), Gribuval and Nielsen \cite{gribonval2003sparse} adopted $l_p$-minimization with $0<p\leq1$,
\begin{eqnarray}
\mathop{\min}\limits_{x \in \mathbb{R}^m} \|x\|_p^p \ s.t. \ Ax=b
\end{eqnarray}
where $\|x\|_p^p=\sum_{i=1}^m |x_i|^p$. In the literature, $\|x\|_p$ is still called $p$-norm of $x$ though it is only a quasi-norm when $0<p<1$ (because in this case it violates the triangular inequality). Due to the fact that $\|x\|_0=\mathop{lim}\limits_{p \to 0} \|x\|_p^p$, it seems to be more natural to consider $l_p$-minimization instead of $l_0$-minimization than others methods and it is important to choose a suitable $p$ in $l_p$-minimization to ensure the solution of $l_p$-minimization can also solve $l_0$-minimization.
\subsection{Related work}

In order to study the equivalence relationship between $l_0$-minimization and $l_p$-minimization, most of related work adopt Restricted Isometry Property (RIP). A matrix $A$ is said to have Restricted Isometry Property (RIP) of order $k$ with Restricted Isometry Constant (RIC) $\delta_k \in (0,1)$, if $\delta_k$ is the smallest constant such that
\begin{eqnarray}\label{RIP}
(1-\delta_k)\|x\|_2 \leq \displaystyle\|Ax\|_2 \leq (1+\delta_k)\|x\|_2
\end{eqnarray}
for every $k$-sparse vector $x$, where a vector $x$ is said $k$-sparse if $\|x\|_0\leq k.$

Cand\`{e}s and Tao \cite{candes2006near,candes2008restricted} showed that every $k$-sparse vector can be recovered via $l_1$-minimization as long as $\delta_{3k}+3\delta_{4k}<2$ or $\delta_{2k}<{\sqrt{2}-1}$. Foucart \cite{foucart2009sparsest} improved the latter inequality and established exact recovery of $k$-sparse vector via $l_1$-minimization under the condition $\delta_{2k}<2(3-\sqrt{2})/7$.
Fourcart \cite{foucart2009sparsest} proved that the condition $\delta_{2k}<0.4531$ can guarantee exact $k$-sparse recovery via $l_p$-minimization for any $0<p<1$. Chartrand \cite{chartrand2007exact} claimed that a $k$-sparse vector can be recovered by $l_p$-minimization for some $p > 0$ small enough provided $\delta_{2k+1} < 1$.

It should be pointed out that these results based on RIP and RIC do not solve this problem completely. It is NP-hard to judge whether nor not a given matrix $A$ satisfies RIP, and it is also NP-hard to get RIC for a given matrix $A$ which is even satisfied with RIP. For a given matrix $A$ satisfied with RIP of order $2k$, it is obvious that $2k<spark(A)$ where $spark(A)$ is the smallest number of columns from $A$ which are linearly dependent, and the results based on RIP only study the case where this unique solution is $k$-sparse with $k<\frac{spark(A)}{2}$, and $l_0$-minimization only has an unique solution. However, we need to realize that the uniqueness assumption is not always certainly tenable. Furthermore, Peng, Yue and Li \cite{peng2015np} have proved that there exists a constant $p(A,b)>0$, such that every a solution of $l_p$-minimization is also the solution of $l_0$-minimization whenever $0<p<p(A,b)$. This result builds a bridge between $l_p$-minimization and $l_0$-minimization, and what is important is that this conclusion is not limited by the uniqueness assumption. However, Peng just proves the existence of such $p$, he does not give us a computable expression of such $p$. Therefore, the purpose of this paper is give a completely answer to this problem for some specific matrices.
\subsection{Main contribution in this paper }

Throughout this paper, for a given vector $\lambda=(\lambda_1,\lambda_2,...,\lambda_n)^T\in \mathbb R^n$ with $\lambda_i\neq0$ $(i\in \{1,2,...,n\})$, we define a Vandermonde matrix $A(m,n,\lambda)\in \mathbb R^{m \times n}$ by
\begin{equation}\label{Van Matrix}
A(m,n,\lambda)=\left(
  \begin{array}{cccc}
    1 & 1 & ... & 1\\
    \lambda_1 & \lambda_2 & ... & \lambda_n\\
    \lambda_1^2 & \lambda_2^2 & ... & \lambda_n^2\\
    \vdots & \vdots& \ddots & \vdots\\
    \lambda_1^{m-1} & \lambda_2^{m-1} & ... & \lambda_n^{m-1}\\
  \end{array}
\right).
\end{equation}
As we know, Vandermonde matrices are widely used in many applications \cite{Marques2016Sampling,Qiao2017Gridless,Fuchs2005Sparsity}. Especially, the sparse recovery methodS of Inverse Synthetic Aperture Radar (ISAR) Imaging adopt the Vandermonde as the measurement matrices \cite{Liu2019A,Wang2018Super}
Another reason why we consider a Vandermonde matrix is due to bounded orthonormal systems (Chapter 11, \cite{foucart2013mathematical}). There are many examples in bound orthonormal systems with the similar structures to a Vandermonde matrix, and one of these important examples is Fourier matrices. In this paper, we assume $x^*$ is one of the solutions of $l_0$-minimization for a given Vandermonde matrix $A(m,n,\lambda)$,
\begin{eqnarray}
\mathop{\min}\limits_{x \in \mathbb{R}^m} \|x\|_0 \ s.t. \ A(m,n,\lambda)x=b,
\end{eqnarray}
and we present an analytic expression of $p^*$ such that $\|x^*\|_p^p<\|x+h\|_p^p$ for any nonzero vector $h\in N(A)$ whenever $0<p<p^*$. 
\subsection{Notation}
For convenience, for $x \in \mathbb R^n$, we define its support by $support\ (x)=\{i:x_i \neq 0\}$ and the cardinality of set S by $|S|$.
Let $Ker(A)=\{x \in \mathbb R^n:Ax=0\}$ be the null space of matrix $A$, denote by $\lambda_{min^+}(A)$ the minimum nonzero absolute-value eigenvalue of $A^TA$ and by $\lambda_{max}(A)$ the maximum one. We also use the subscript notation $x_S$ to denote such a vector that is equal to $x$ on the index set $S$ and zero everywhere else.

In this paper, we define a Vandermonde matrix by (\ref{Van Matrix}), and use $A$ represents an ordinary underdetermined matrix. We denote the smallest number of columns from $A$ that are linearly dependent by $spark(A)$.
\section{Preliminaries}
In this section, we will focus on introducing some lemmas and definitions. In a mathematical sense, the left side of the inequality in the definition of RIP (\ref{RIP}) is crucial, especially in the proof of a large number of theorems in sparse representation theory. However, we do not adopt RIP (\ref{RIP}) in this paper because it can not be applied to the case where $\|x\|_0>\frac{spark(A)}{2}$, we need to change this equality into a new form which seems to be more reasonable. 

Usually, for an underdetermined equations we can not conclude that
\begin{eqnarray}{\label{21}}
\lambda_{min^+}(A)\|x\|_2^2\leq\|Ax\|_2^2,
\end{eqnarray}
for every vector $x\notin N(A)$.
Because it is obvious that  
\begin{eqnarray}
\lim \limits_{t\rightarrow +\infty} \frac{\|A(\tilde{x}+t\cdot\tilde{h}\|_2)}{\|\tilde{x}+t\cdot\tilde{h}\|_2}=0
\end{eqnarray}
for any vector $\tilde{x}\notin N(A)$ and a vector $\tilde{h}\in N(A)$. In order to overcome this difficulty, we will study the condition under which the inequality (\ref{21}) is satisfied. We first present a obvious fact in the following lemma and give a simple proof.
\begin{lemma}\label{L1}
Given an underdetermined matrix $A \in \mathbb{R}^{m \times n}$, then there exist two constants $0<u\leq w$ with
\begin{eqnarray}
0<\lambda_{min^{+}}(A)\leq u^2\leq w^2 \leq \lambda_{max}(A),
\end{eqnarray}
such that
\begin{eqnarray}
 u^2\|x\|_2^2 \leq \|Ax\|_2^2 \leq w^2\|x\|_2^2,
\end{eqnarray}
holds for every $x\in \mathbb{R}^n$ with $\|x\|_0<spark(A)$.
\end{lemma}

\begin{proof}
The proof is divided into two steps.

Step 1: To prove the existence of $u$.

In order to prove this result we just need to prove that the set
\begin{eqnarray}
V=\{u:  \|Ax\|_2/\|x\|_2 \geq u,\  for\ any\ nonzero\ x \ with\ \|x\|_0 \leq spark(A) \}
\end{eqnarray}
has a nonzero infimum.

We assume that inf $V$=0, i.e., for any $n \in N^{+}$, there exists a vector $\|x_n\|_0 \leq spark(A)$ such that
\begin{eqnarray}
\|Ax_n\|_2/\|x_n\|_2 \leq n^{-1}
\end{eqnarray}
Without of generality, we can assume that $\|x_n\|_2=1$, so the bounded sequence $\{x_n\}$ has a subsequence $\{x_{n_i}\}$ which is convergent, i.e. $x_{n_i} \to x_0$ and it is obvious that $Ax_0=\bf{0}$ because that the function $y(x)=Ax$ is a continuous one.

Let $J(x_0)=\{i:(x_0)_i \ne 0\}$, since $x_{n_i} \to x_0$, it is easy to get that, for any $i \in J(x_0)$, there exists $N_i$ such that $(x_{n_k})_i \ne 0$ when $k \geq N_i$.

Let $ N=\max \limits_{i \in J(x_0)}N_i$, for any $i \in J(x_0)$, it is easy to get that $(x_{n_k})_i \ne 0$ when $k \geq N$. Therefore, we can get that $\|x_{n_k}\|_0 \geq \|x_0\|_0$ when $k \geq N$, such that $\|x_0\|_0 \leq spark(A)$.

Therefore, there exists a constant $u>0$ such that $\|Ax\|_2 \geq u\|x\|_2$, for any $x \in \mathbb{R}^n$ with $\|x\|_0 \leq spark(A)$.

Step 2: To prove $u^2 \geq \lambda_{min^{+}}(A^TA)$.

According to the proof above, there exists a vector $\widetilde x \in \mathbb{R}^n$ with $\|\widetilde x\|_0 \leq spark(A)$ such that $\|A\widetilde x\|_2=u\| \widetilde x\|_2$.

Let $V=support (\widetilde x)$, it is easy to get that
\begin{eqnarray}
u^2x^Tx \leq x^TA_V^TA_Vx,
\end{eqnarray}
for all $x \in \mathbb{R}^{|V|}$. Therefore, the smallest eigenvalue of $A_V^TA_V$  is $u^2$ since $A_V^TA_V \in R^{|V| \times |V|}$ is a symmetric matrix. and we can choose an eigenvector $z\in R^{|V|}$ of eigenvalue $u^2$.

If $u^2<\lambda_{min^{+}}(A^TA)$, then we can consider such a vector $x^{'}\in \mathbb{R}^n$ with $x_i^{'}=z_i$ when $i \in V$ and zero everywhere else. Therefore, it is easy to get that $A^TAx^{'}=u^2x^{'}$ which contradicts the definition of $\lambda_{min^{+}}(A^TA)$.

At last, we notice that $A^TA$ is a semi-positive definite matrix, such that $\|Ax\|_2^2=x^TA^TAx\leq \lambda_{max}(A^TA)\|x\|_2^2$ for all $x\in \mathbb{R}^n$.
Therefore, the proof is completed.
\end{proof}
The following lemma introduced by Foucart \cite{Foucart2012Sparse} is an important inequality to compare different norms between different subvectors.
\begin{lemma}{\rm \cite{Foucart2012Sparse}}\label{L2}
If $0<p<q$, and $u_1\geq \ldots \geq u_k\geq u_{k+1}\ldots \geq u_s \geq u_{s+1}\ldots \geq u_{k+t}\geq 0$, it holds that
\begin{eqnarray}
\displaystyle\left(\sum_{i=k+1}^{k+t}u_i^q\right)^{\frac{1}{q}}\leq C_{p,q}(k,s,t)\left(\sum_{i=1}^s u_i^p\right)^{\frac{1}{p}}
\end{eqnarray}
with $\displaystyle C_{p,q}(k,s,t)=max\left\{\frac{t^{\frac{p}{q}}}{s},\left(\frac{p}{q}\right)^{\frac{p}{q}}\left(1-\frac{p}{q}\right)^{1-\frac{p}{q}}k^{\frac{p}{q}-1}\right\}^{\frac{1}{p}}$
\end{lemma}
\begin{lemma}\label{L3}
For $p\in (0,1]$, we have that
\begin{eqnarray}
\displaystyle \left(\frac{p}{2}\right)^{\frac{1}{2}}\left(\frac{1}{2-p}\right)^{\frac{1}{2}-\frac{1}{p}}\geq \frac{\sqrt{2}}{2}
\end{eqnarray}

\end{lemma}
\begin{proof}
We denote a function $f(p)$ on interval $(0,1]$
\begin{eqnarray}
f(p)=\displaystyle \left(\frac{p}{2}\right)^{\frac{1}{2}}\left(\frac{1}{2-p}\right)^{\frac{1}{2}-\frac{1}{p}},
\end{eqnarray}
and we can get that
\begin{eqnarray}
\displaystyle\ln f(p)=\frac{1}{2}\ln\frac{p}{2}-\left(\frac{1}{2}-\frac{1}{p}\right)\ln(2-p).
\end{eqnarray}
It is easy to get that
\begin{eqnarray}
\notag h(p)&=&\frac{f'(p)}{f(p)}=\frac{1}{2p}-\left(\frac{1}{p^2}\ln(2-p)-\left(\frac{1}{2}-\frac{1}{p}\right)\frac{1}{2-p}\right),\\
&=&-\frac{1}{p^2}\ln(2-p)\leq 0.
\end{eqnarray}
Therefore, $f(p)$ is nondecreasing in $p \in (0,1]$, and we can get $f(p)\geq f(1)=\frac{\sqrt{2}}{2}$.

The proof is completed.
\end{proof}
As an important matrix used in the practical application, there are a lot of properties of Vandermonde matrices, and the following lemma presents an important property which is widely used in this paper.
\begin{lemma}{\rm (Theorem A.25 in  \cite{foucart2013mathematical})}\label{L4}
If $\lambda_1>\lambda_2>...\lambda_n>0$, then the Vandermonde matrix $A(m,n,\lambda)$ is totally positive, i.e, for any sets $I, J\subset\{1,2,3,...,n\}$ of equal size,
\begin{eqnarray}
det A(m,n,\lambda)_{I,J}>0,
\end{eqnarray}
where $A(m,n,\lambda)_{I,J}$ is the submatrix of $A(m,n,\lambda)$ with rows and columns indexed by $I$ and $J$.
\end{lemma}
\begin{corollary}\label{C1}
If $|\lambda_i|\neq |\lambda_j|$ as long as $i\neq j$ and $|\lambda_i|\neq 0$ $(i,j\in\{1,2,3,...,n\})$, the submatrix $A(m,n,\lambda)_{I,J}$ is an an invertible matrix for any sets $I, J\subset\{1,2,3,...,n\}$ of equal size.
\end{corollary}
\section{Main Contribution}

For a given Vandermonde matrix $A(m,n,\lambda)\in \mathbb R^{m \times n}$, $x^*$ is the sparest solution of $A(m,n,\lambda)x=b$, and it is obvious that $1\leq\|x^*\|_0\leq m$ since $rank(A(m,n,\lambda))=m$. Therefore, the question is: what is the $p$ such that $\|x^*\|_p^p<\|x^*+h\|_p^p$ for every nonzero vector $h\in N(A(m,n,\lambda))$?

In order to answer this question completely, we consider the following two cases in this paper,

Case \uppercase\expandafter{\romannumeral1}, $1\leq\|x^*\|_0<\frac{spark(A(m,n,\lambda))}{2}=\frac{m+1}{2}$.

Case \uppercase\expandafter{\romannumeral2}, $\frac{spark(A(m,n,\lambda))}{2}=\frac{m+1}{2}\leq\|x^*\|_0\leq m$.

In Section 3.1, we will consider Case \uppercase\expandafter{\romannumeral1}, and we will give a result with a wider applicability which not only can be applied to an ordinary underdetermined matrix. In Section 3.2, we will consider Case \uppercase\expandafter{\romannumeral2} for a Vandermonde matrix. In order to describe our results more clearly, we define a matrix function for a given underdetermined matrix $A$,
\begin{eqnarray}\label{Function}
p^*(A):=min\left\{1,\frac{16\lambda_{min^+}(A)^2}{(\sqrt 2+1)^2(\lambda_{max}(A)-\lambda_{min^+}(A))^2}\right\}.
\end{eqnarray}
\subsection{Case \uppercase\expandafter{\romannumeral1}}
In this subsection, we consider a more general sense, i.e., the matrix $A$ is an arbitrary underdetermined matrix and $\|\breve{x}\|_0<\frac{spark(A)}{2}$ is the solution of $l_0$-minimization. The following theorem presents us such a $p$ for this situation.
\begin{theorem}
Given an underdetermined matrix $A\in \mathbb{R}^{m \times n}$ with $m \leq n$, and $\breve{x}$ is the solution of the following $l_0$-minimization,
\begin{eqnarray}
\mathop{\min}\limits_{x \in \mathbb{R}^n} \|x\|_0\ s.t. \ Ax=b
\end{eqnarray}
If $\|\breve{x}\|_0=k< \frac{spark(A)}{2}$, then we have that
\begin{eqnarray}
\|\breve{x}\|_p^p<\|\breve{x}+h\|_p^p
\end{eqnarray}
for any nonzero vector $h\in N(A)$ and $0<p<p(A)$, where the matrix function $p^*(A)$ is defined by (\ref{Function}).
\end{theorem}
\begin{proof}
For a underdetermined matrix $A\in \mathbb{R}^{m \times n}$ with $m \leq n$, $\breve{x}$ is the solution of the $l_0$-minimization and $\|\breve{x}\|_0=k<\frac{spark(A)}{2}$.

Now, for a vector $h\in N(A)$, we consider the index set $S_0=support(\breve{x}).$

$S_1$=\{ indices of the largest $k$ values component of $h$ except $S_0$\}.

$S_2$=\{ indices of the largest $k$ values component of $h$ except $S_0$ and $S_1$\}.

\dots

$S_t$=\{ indices of the rest components of $h$ \}.

Before we start our main proof, we present a simple inequality which is useful in our proof.

For any vectors $x_1$ and $x_2$, $\|x_i\|_0 < \frac{spark(A)}{2}\ (i=1,2)$ and $support(x_1) \cap support(x_2)= \varnothing$, then we have that
\begin{eqnarray}\label{BU}
\displaystyle|\langle Ax_1,Ax_2 \rangle | \leq \frac{\lambda_{max}(A)-\lambda_{min^+}(A)}{2}\|x_1\|_2\|x_2\|_2.
\end{eqnarray}

By Lemma \ref{L1}, it is easy to get that
\begin{eqnarray}
\notag\frac{|\langle Ax_1,Ax_2 \rangle|}{\|x_1\|_2\|x_2\|_2}&=&\displaystyle\left|\left\langle A\left(\frac{x_1}{\|x_1\|_2}\right), A\left(\frac{x_2}{\|x_2\|_2}\right)\right\rangle\right|,\\
\notag &=&\frac{1}{4}\left|\left\|A\left(\frac{x_1}{\|x_1\|_2}+\frac{x_2}{\|x_2\|_2}\right)\right\|_2^2-\left\|A\left(\frac{x_1}{\|x_1\|_2}-\frac{x_2}{\|x_2\|_2}\right)\right\|_2^2\right|,\\
&\leq&\frac{1}{4}\left|w^2\left\| \frac{x_1}{\|x_1\|_2}+\frac{x_2}{\|x_2\|_2}\right\|_2^2-u^2\left\| \frac{x_1}{\|x_1\|_2}-\frac{x_2}{\|x_2\|_2}\right\|_2^2 \right|.
\end{eqnarray}

Since $support(x_1) \cap support(x_2)= \varnothing$, we have that
\begin{eqnarray}
\left\|\frac{x_1}{\|x_1\|_2}+\frac{x_2}{\|x_2\|_2}\right\|_2^2=\left\|\frac{x_1}{\|x_1\|_2}-\frac{x_2}{\|x_2\|_2}\right\|_2^2=2,
\end{eqnarray}
from which we get that
\begin{eqnarray}
\displaystyle|\langle Ax_1,Ax_2 \rangle | \leq \frac{w^2-u^2}{2}\|x_1\|_2\|x_2\|_2.
\end{eqnarray}
Furthermore, we can get that
\begin{eqnarray}
\displaystyle|\langle Ax_1,Ax_2 \rangle | \leq \frac{\lambda_{max}(A)-\lambda_{min^+}(A)}{2}\|x_1\|_2\|x_2\|_2.
\end{eqnarray}

By Lemma \ref{L1} and the inequality (\ref{BU}) it is obvious that
\begin{eqnarray}\label{E1}
\notag \|h_{S_0}\|_2^2+\|h_{S_1}\|_2^2 &=& \|h_{S_0}+h_{S_1}\|_2^2 \\
\notag &\leq & \frac{1}{\lambda_{min^+}(A)}\|Ah_{S_0}+Ah_{S_1}\|_2^2\\
\notag &\leq & \frac{1}{\lambda_{min^+}(A)}\left(\sum_{i=2}^t \langle -Ah_{S_0},Ah_{S_i}\rangle+\sum_{i=2}^t \langle -Ah_{S_1},Ah_{S_i}\rangle\right)\\
&\leq & \frac{\lambda_{max}(A)-\lambda_{min^+}(A)}{2\lambda_{min^+}(A)}(\|h_{S_0}\|_2+\|h_{S_1}\|_2)\left(\sum_{i=2}^t \|h_{S_i}\|_2\right).
\end{eqnarray}
For any $2\leq i\leq t$, it is easy to get that
\begin{eqnarray}
\|h_{S_i}\|_2+\|h_{S_{i+1}}\|_2+\|h_{S_{i+2}}\|_2+\|h_{S_{i+3}}\|_2 \leq 2\|h_{S_i\bigcup S_{i+1}\bigcup S_{i+2} \bigcup S_{i+3}}\|_2.
\end{eqnarray}
Therefore,
\begin{eqnarray}
\sum_{i\geq2}\|h_{S_i}\|_2\leq 2 \sum_{i\geq0}\left(\|h_{S_{4i+2}\bigcup S_{4i+3}\bigcup S_{4i+4} \bigcup S_{4i+5}}\|_2\right)
\end{eqnarray}

By Lemma \ref{L3}, Since we can get that
\begin{eqnarray}
\displaystyle \left(\frac{p}{2}\right)^{\frac{1}{2}}\left(\frac{1}{2-p}\right)^{\frac{1}{2}-\frac{1}{p}}\geq \frac{\sqrt{2}}{2},
\end{eqnarray}
Therefore, we have that
\begin{eqnarray}
\displaystyle \left(\frac{p}{2}\right)^{\frac{1}{2}}\left(\frac{1}{2-p}\right)^{\frac{1}{2}-\frac{1}{p}}\geq 2^{\frac{1}{2}-\frac{1}{p}},
\end{eqnarray}
By Lemma \ref{L2}, we take $q=2$, $s=t=4k$, and $0<p\leq1$. It is easy to get that
\begin{eqnarray}
\|h_{S_i\bigcup S_{i+1}\bigcup S_{i+2} \bigcup S_{i+3}}\|_2\leq C(p) \|h_{S_{i-1}\bigcup S_{i}\bigcup S_{i+1} \bigcup S_{i+2}}\|_p,
\end{eqnarray}
where
\begin{eqnarray}
\notag C(p)&=&max\{(4k)^{\frac{1}{2}-\frac{1}{p}},\left(\frac{p}{2}\right)^{\frac{1}{2}}\left(2-p\right)^{\frac{1}{p}-\frac{1}{2}}(2k)^{\frac{1}{2}-\frac{1}{p}}\}\\
&=&\left(\frac{p}{2}\right)^{\frac{1}{2}}\left(2-p\right)^{\frac{1}{p}-\frac{1}{2}}(2k)^{\frac{1}{2}-\frac{1}{p}}.
\end{eqnarray}

Therefore, we can get that
\begin{eqnarray}\label{E2}
\notag \sum_{i=2}^{t}\|h_{S_i}\|_2 &\leq& 2 \sum_{i\geq 0}(\|h_{S_{4i+2}\bigcup S_{4i+3}\bigcup S_{4i+4} \bigcup S_{4i+5}}\|_2)\\
\notag &\leq & 2 \cdot \left(\frac{p}{2}\right)^{\frac{1}{2}}\left(2-p\right)^{\frac{1}{p}-\frac{1}{2}}(2k)^{\frac{1}{2}-\frac{1}{p}}\sum_{i\geq 0}(\|h_{S_{4i+1}\bigcup S_{4i+2}\bigcup S_{4i+3} \bigcup S_{4i+4}}\|_p)\\
&\leq & 2 \cdot \left(\frac{p}{2}\right)^{\frac{1}{2}}\left(2-p\right)^{\frac{1}{2}-\frac{1}{p}}(2k)^{\frac{1}{p}-\frac{1}{2}}\|h_{S_0^C}\|_p.
\end{eqnarray}
Substituting the inequality (\ref{E2}) into (\ref{E1}), we have that
\begin{eqnarray}\label{E3}
\notag \|h_{S_0}\|_2^2+\|h_{S_1}\|_2^2 &\leq& \frac{\lambda_{max}(A)-\lambda_{min^+}(A)}{2\lambda_{min^+}(A)}(\|h_{S_0}\|_2+\|h_{S_1}\|_2)\left(\sum_{i=2}^t \|h_{S_i}\|_2\right)\\
&\leq & B\cdot(\|h_{S_0}\|_2+\|h_{S_1}\|_2),
\end{eqnarray}
where $B=\frac{\lambda_{max}(A)-\lambda_{min^+}(A)}{\lambda_{min^+}(A)}\left(\frac{p}{2}\right)^{\frac{1}{2}}\left(2-p\right)^{\frac{1}{p}-\frac{1}{2}}(2k)^{\frac{1}{2}-\frac{1}{p}}\|h_{S_0^C}\|_p$.

By (\ref{E3}), we can get that
\begin{eqnarray}
(\|h_{S_0}\|_2-\frac{B}{2})^2+(\|h_{S_1}\|_2-\frac{B}{2})^2\leq \frac{B^2}{2}
\end{eqnarray}
Therefore, we can get that
\begin{eqnarray}
\|h_{S_0}\|_2\leq \frac{\sqrt 2+1}{2} B.
\end{eqnarray}
By H\"older's inequality, for any vector $x$, we can get that
\begin{eqnarray}
\notag\|x\|_p^p\leq\|x\|_0^{1-\frac{p}{2}}\left(\|x\|_2^2\right)^{\frac{p}{2}}=\|x\|_0^{1-\frac{p}{2}} \|x\|_2^p.
\end{eqnarray}
Therefore, we can get that
\begin{eqnarray}
\|h_{S_0}\|_p\leq k^{\frac{1}{p}-\frac{1}{2}}\|h_{S_0}\|_2\leq k^{\frac{1}{p}-\frac{1}{2}}\cdot\frac{\sqrt 2+1}{2} B.
\end{eqnarray}

We define a function $\varphi (p)$ on interval $(0,1]$ by $\varphi (p)=\left(1-\frac{p}{2}\right)^{\frac{1}{p}-\frac{1}{2}}$, and it is easy to get that
\begin{eqnarray}
\notag \frac{\varphi '(p)}{\varphi (p)}&=&-\frac{1}{p^2}\ln (1-\frac{p}{2})+(\frac{1}{p}-\frac{1}{2})\frac{-1}{2-p}\\
\notag &=&\frac{1}{p}\left(\frac{1}{p}\ln\frac{2}{2-p}-\frac{1}{2}\right)\\
\notag &=&\frac{1}{p}\left(\frac{1}{p}\ln\left(1+\frac{p}{2-p}\right)-\frac{1}{2}\right)\\
&>&\frac{1}{p}\left(\frac{1}{p}\cdot\frac{p}{2-p}-\frac{1}{2}\right)>0,
\end{eqnarray}
Because $\lim \limits_{p\rightarrow 0} \varphi (p)=e^{-\frac{1}{2}}$ and $\varphi (1)=\frac{\sqrt{2}}{2}$ . Therefore, we can get that $\varphi (p)\leq \frac{\sqrt{2}}{2}$ for $0<p\leq 1$.

We consider the following inequality,
\begin{eqnarray}\label{E4}
\frac{\sqrt 2+1}{2}\cdot \frac{\lambda_{max}(A)-\lambda_{min^+}(A)}{\lambda_{min^+}(A)}\cdot \frac{\sqrt{2}}{2}\cdot\sqrt{\frac{p}{2}}<1.
\end{eqnarray}
It is easy to get the solution of the inequality (\ref{E4})
\begin{eqnarray}
0<p<p^*(A)=min\left\{1,\frac{16\lambda_{min^+}(A)^2}{(\sqrt 2+1)^2(\lambda_{max}(A)-\lambda_{min^+}(A))^2}\right\}.
\end{eqnarray}
We notice that
\begin{eqnarray}
\notag\|h_{S_0}\|_p&<&k^{\frac{1}{p}-\frac{1}{2}}\|h_{S_0}\|_2\leq k^{\frac{1}{p}-\frac{1}{2}}\cdot\frac{\sqrt 2+1}{2} B\\
&\leq & \frac{\sqrt 2+1}{2}\cdot \frac{\lambda_{max}(A)-\lambda_{min^+}(A)}{\lambda_{min^+}(A)}\cdot \frac{\sqrt{2}}{2}\cdot\sqrt{\frac{p}{2}}\|h_{S_0^C}\|_p.
\end{eqnarray}
It is obvious that $\|h_{S_0}\|_p<\|h_{S_0^C}\|_p$ when $0<p<p^*(A)$, and we can get that
\begin{eqnarray}
\|\breve{x}+h\|_p^p=\|\breve{x}+h_{S_0}\|_p^p+\|h_{S_0^C}\|_p^p\geq \|\breve{x}\|_p^p-\|h_{S_0}\|_p^p+\|h_{S_0^C}\|_p^p >\|\breve{x}\|_p^p.
\end{eqnarray}

The proof is completed
\end{proof}

By Theorem 1, we can take $A=A(m,n,\lambda)$, and $\breve{x}=x^*$ then we can get the following corollary bacause $spark(A(m,n,\lambda))=m+1$.
\begin{corollary}
Given an Vandermonde matrix $A(m,n,\lambda)\in \mathbb{R}^{m \times n}$ with $m \leq n$, and $x^*$ is the solution of the following $l_0$-minimization,
\begin{eqnarray}
\mathop{\min}\limits_{x \in \mathbb{R}^n} \|x\|_0\ s.t. \ A(m,n,\lambda)x=b
\end{eqnarray}
If $\|x^*\|_0< \frac{m+1}{2}$, then we have that
\begin{eqnarray}
\|x^*\|_p^p<\|x^*+h\|_p^p
\end{eqnarray}
for any nonzero vector $h\in N(A(m,n,\lambda))$ and $0<p<p^*(A(m,n,\lambda))$ where the matrix function $p^*(A)$ is defined by (\ref{Function}).
\end{corollary}
\begin{remark}
In the proof of Theorem 1, the condition $\|\breve{x}\|_0\leq spark(A)/2$ means that $\breve{x}$ is the unique solution of $l_0$-minimization with $b=A\breve{x}$.
\end{remark}
\begin{remark}
We should realize that the conclusion in Theorem 1 considers every vector $x$ with $\|x\|_0<\frac{spark(A)}{2}$. Recall the results based on RIP and RIC, these work only consider $k$-spare solution when the matrix $A$ satisfied RIP of order $2k$, and it is obvious that $k<\frac{spark(A)}{2}$ for such a matrix $A$, but the opposite is not always true. For these reasons, Theorem 1 has a wider applicability than others.
\end{remark}
Theorem 1 is the first main result in this paper, we present an analytic expression of $p$ such that $\|x^*\|_p^p<\|x^*+h\|_p^p$ for any $h\in N(A)$. Theorem 1 also is the basis for other theorems in this paper. We notice that the condition $\|x^*\|_0<\frac{spark(A)}{2}$ is  vital to the proof of Theorem 1, so we consider to construct new matrices based on the old one in the following subsection, and these new matrices have a bigger spark number.
\subsection{Case \uppercase\expandafter{\romannumeral2}}
In this subsection, we will consider the situation where $\frac{m+1}{2}\leq\|x^*\|_0\leq m$. We should realize that this situation is completely different from the situation in Section 3.1, because the matrix $A(m,n,\lambda)$ does not have the ability to recover every $k$-sparse vector when $\frac{m+1}{2}\leq k \leq m$. So the conclusions in Section 3.1 can not be true in Case \uppercase\expandafter{\romannumeral2}.

However, the method in Section 3.1 provides us a way to solve this situation, i,e., we may construct a new matrix composed of $A(m,n,\lambda)$, and the new matrix has the ability to recover every $m$-sparse vector.

For a given Vandermonde matrix $A(m,n,\lambda)\in \mathbb{R}^{m\times n}$, we define a series of matrices $A^{(t)}(m,n,\lambda,x_t,y_t)\in \mathbb{R}^{(2m+2)\times (m+n+2)}$,
\begin{eqnarray}\label{At}
A^{(t)}(m,n,\lambda,x_t,y_t)=\left(
\begin{array}{ccccccccc}
1 & 1 & 1 & \dots & 1 & 0 & 0 & 0& 0\\
\lambda_1 & \lambda_2 & \lambda_3 & \dots & \lambda_n & 0 & 0 & 0 & 0\\
\lambda_1^2& \lambda_2^2 & \lambda_3^2 & \dots & \lambda_n^2 & 0 & 0 & 0 & 0\\
\vdots & \vdots & \vdots & \ddots & \vdots & 0 & 0 & 0 & 0\\
\lambda_1^{m-1} & \lambda_2^{m-1} & \lambda_3^{m-1} & \dots & \lambda_n^{m-1} & 0 & 0 & \ldots & 0 \\
x_t\cdot\lambda_1^m & x_t\cdot\lambda_2^m & x_t\cdot\lambda_3^m & \dots & x_t\cdot\lambda_n^m & 1 & 0 & \ldots & 0 \\
y_t\cdot\lambda_1^{m+1} & y_t\cdot\lambda_2^{m+1} & y_t\cdot\lambda_3^{m+1} & \dots & y_t\cdot\lambda_n^{m+1} & 0 & 1 & 0 & 0 \\
\vdots & \vdots & \vdots & \ddots & \vdots & \vdots & \vdots & \ddots & \vdots \\
y_t\cdot\lambda_1^{2m+1} & y_t\cdot\lambda_2^{2m+1} & y_t\cdot\lambda_3^{2m+1} & \dots & y_t\cdot\lambda_n^{2m+1} & 0 & 0 & \ldots & 1 \\
\end{array}
\right),
\end{eqnarray}
where the sequences $\lim \limits_{t\rightarrow\infty} x_t=0$, $\lim \limits_{t\rightarrow\infty} y_t=0$, $x_t>0$, and $y_t>0$.

Furthermore, we also define a new matrix $A^{(0)}(m,n,\lambda)\in \mathbb{R}^{(2m+2)\times (n+m+2)}$
\begin{eqnarray}\label{At0}
A^{(0)}(m,n,\lambda)=\left(
\begin{array}{ccccccccc}
1 & 1 & 1 & \dots & 1 & 0 & 0 & 0& 0\\
\lambda_1 & \lambda_2 & \lambda_3 & \dots & \lambda_n & 0 & 0 & 0 & 0\\
\lambda_1^2& \lambda_2^2 & \lambda_3^2 & \dots & \lambda_n^2 & 0 & 0 & 0 & 0\\
\vdots & \vdots & \vdots & \ddots & \vdots & 0 & 0 & 0 & 0\\
\lambda_1^{m-1} & \lambda_2^{m-1} & \lambda_3^{m-1} & \dots & \lambda_n^{m-1} & 0 & 0 & \ldots & 0 \\
0 & 0& 0 & \dots & 0 & 1 & 0 & \ldots & 0 \\
0 & 0 & 0 & \dots & 0 & 0 & 1 & 0 & 0 \\
\vdots & \vdots & \vdots & \ddots & \vdots & \vdots & \vdots & \ddots & \vdots \\
0 & 0 & 0 & \dots & 0 & 0 & 0 & \ldots & 1 \\
\end{array}
\right).
\end{eqnarray}

As we have mentioned, the new matrices $A^{(0)}(m,n,\lambda)$ and $A^{(t)}(m,n,\lambda,x_t,y_t)$ based on $A(m,n,\lambda)$ are the key to our result. The following lemma presents us an important property of $A^{(t)}(m,n,\lambda,x_t,y_t)$ which is also important.
\begin{proposition}\label{P1}
If the Vandermonde matrix $A(m,n,\lambda)$ with $|\lambda_i|\neq |\lambda_j|$ ($i\neq j \ and\ i,j\in \{1,2,...,n\}$). Recall the definitions of $A^{(t)}(m,n,\lambda,x_t,y_t)$ $(\ref{At})$, then we have that
\begin{eqnarray}
spark(A^{(t)}(m,n,\lambda,x_t,y_t))=2m+3.
\end{eqnarray}
\end{proposition}
\begin{proof}
It is easy to get that $rank(A^{(t)}(m,n,\lambda,x_t,y_t))=2m+2$, however it does not mean that $spark(A^{(t)}(m,n,\lambda,x_t,y_t))=2m+3$ since $A^{(t)}(m,n,\lambda,x_t,y_t)$ has a different structure from $A(m,n,\lambda)$.

By the definition of $spark(A^{(t)}(m,n,\lambda,x_t,y_t)$, it is obvious that 
\begin{eqnarray}
spark(A^{(t)}(m,n,\lambda,x_t,y_t)=spark(A^*), 
\end{eqnarray}
where
\begin{center}
$A^*=\left(
\begin{array}{ccccccccc}
1 & 1 & 1 & \dots & 1 & 0 & 0 & 0& 0\\
\lambda_1 & \lambda_2 & \lambda_3 & \dots & \lambda_n & 0 & 0 & 0 & 0\\
\lambda_1^2& \lambda_2^2 & \lambda_3^2 & \dots & \lambda_n^2 & 0 & 0 & 0 & 0\\
\vdots & \vdots & \vdots & \ddots & \vdots & 0 & 0 & 0 & 0\\
\lambda_1^{m-1} & \lambda_2^{m-1} & \lambda_3^{m-1} & \dots & \lambda_n^{m-1} & 0 & 0 & \ldots & 0 \\
\lambda_1^m & \lambda_2^m & \lambda_3^m & \dots & \lambda_n^m & 1 & 0 & \ldots & 0 \\
\lambda_1^{m+1} & \lambda_2^{m+1} & \lambda_3^{m+1} & \dots & \lambda_n^{m+1} & 0 & 1 & 0 & 0 \\
\vdots & \vdots & \vdots & \ddots & \vdots & \vdots & \vdots & \ddots & \vdots \\
\lambda_1^{2m+1} & \lambda_2^{2m+1} & \lambda_3^{2m+1} & \dots & \lambda_n^{2m+1} & 0 & 0 & \ldots & 1 \\
\end{array}
\right).$
\end{center}

Let $S_1=\{1,2,...,n\}$ and $S_2=\{n+1,n+2,...,n+m+2\}$. In order to find the smallest number of columns from $A^*$ which are linearly dependent, we assume that the index set $S^*$ corresponds to $spark(A^*)$. i.e., the columns of $A_{S^*}$ are linearly dependent and $|S^*|=spark(A^*)$.

Let $S^*=S^*_1\bigcup S^*_2$ with $S^*_1\subseteq S_1$ and $S^*_2\subseteq S_2$. It is obvious that $m+1\leq|S^*_1|\leq 2m+3$.
We define some vectors $B^T_i$ ($i\in \{1,2,...,m+2\}$) by
\begin{equation}\label{Bvectors}
\begin{aligned}
&B^T_1=(\lambda_1^m,\lambda_2^m,\cdots ,\lambda_n^m),\\
&B^T_2=(\lambda_1^{m+1},\lambda_2^{m+1},\cdots ,\lambda_n^{m+1}),\\
& \ldots \\
&B^T_{m+2}=(\lambda_1^{2m+1},\lambda_2^{2m+1},\cdots ,\lambda_n^{2m+1}).
\end{aligned}
\end{equation}

We notice that the submatrix $[A^T(m,n,\lambda), B_1, B_2,...,B_{m+2}]^T$ is column full rank and the vectors $B_i$ $(i=1,2,...,m+2)$ are linearly independent.
By Lemma \ref{L4} and Corollary \ref{C1}, the submatrix $[A^T(m,n,\lambda)_{S^*_1}, B_{i_1}, B_{i_2},...,B_{i_{|S^*|-m}}]^T$ is an invertible matrix, where $i_1, i_2,...,i_{|S^*_1|-m}$ are $|S^*_1|-m$ different numbers from $\{1,2,3,...,m+2\}$.

It is easy to get that we can choose a vector $h\in N(A(m,n,\lambda))$ with $\|h\|_0=m+1$, and the set $support(h)\cup S_2$ is one of the methods to get the smallest number of columns from $A^*$ that are linearly dependent. Therefore, $|S_1^*|=m+1$ and $|S_2^*|=m+2$, and we have that $spark(A^*)= 2m+3$.
\end{proof}
\begin{theorem}
For a given Vandermonde matrix $A(m,n,\lambda)\in \mathbb{R}^{m\times n}$, $|\lambda_i|\neq |\lambda_j|$ with ($i,j\in\{1,2,3,...,n\}, i\neq j$). Recall the definitions of $A^{(0)}(m,n,\lambda)$ (\ref{At0}) and $p^*(A)$ (\ref{Function}).
If the solution $x^*$ of $l_0$-minimization with $\frac{m+1}{2}\leq\|x^*\|_0\leq m$, then we have that
\begin{eqnarray}
\|x^*\|_p^p\leq\|x^*+h\|_p^p,
\end{eqnarray}
for every nonzero vector $h\in N(A(m,n,\lambda))$ and $0<p<p^*(A^{(0)}(m,n,\lambda))$.
\end{theorem}
\begin{proof}
First of all, we will prove a simple fact that
\begin{equation*}
\lim \limits_{t\rightarrow \infty} p^*(A^{(t)}(m,n,\lambda,x_t,y_t))=p^*(A^{(0)}(m,n,\lambda))
\end{equation*}
It is obvious that $rank(A^{(t)}(m,n,\lambda,x_t,y_t))=2m+3$ for any $t$, and the characteristic polynomial of $A^{(t)}(m,n,\lambda,x_t,y_t)^TA^{(t)}(m,n,\lambda,x_t,y_t)$ always only has $2m+3$ positive roots. So we can conclude that $\lim \limits_{t\rightarrow \infty} \lambda_{max}(A^{(t)}(m,n,\lambda,x_t,y_t))=\lambda_{max}(A^{(0)}(m,n,\lambda))$. and $\lim \limits_{t\rightarrow \infty} \lambda_{min^+}(A^{(t)}(m,n,\lambda,x_t,y_t))=\lambda_{min^+}(A^{(0)}(m,n,\lambda))$.
Therefore, we can get that
\begin{eqnarray}
\lim \limits_{t\rightarrow \infty} p^*(A^{(t)}(m,n,\lambda,x_t,y_t))=p^*(A^{(0)}(m,n,\lambda)).
\end{eqnarray} 

Next, we will prove that 
\begin{eqnarray}
\|x^*\|_{\check{p}}^{\check{p}}<\|x^*+\bar{h}\|_{\check{p}}^{\check{p}}
\end{eqnarray}
for a fixed nonzero vector $\bar{h}\in N(A(m,n,\lambda))$ and a fixed constant $\check{p}\in(0,p^*(A^{(0)}(m,n,\lambda))$.

Recall the definition of $B^T_i$ $(\ref{Bvectors})$, and we define $l_i$ by 
\begin{eqnarray}
l_i=<B_i,\bar{h}>
\end{eqnarray}
for $i=1,2,\ldots m+4$. Without loss of generality, we can assume that $|l_1|\geq|l_2|\geq \ldots |l_{m+2}|$.
Furthermore, we can define $\{x_t\}$ and $\{y_t\}$ by
\begin{eqnarray}\label{xt}
x_t=(m+1)^{\frac{1}{\check{p}}}|l_1|^{-1}\frac{1}{t},
\end{eqnarray}
and
\begin{eqnarray}\label{yt}
y_t=|l_2|^{-1}\frac{1}{t}.
\end{eqnarray}

Since $\lim \limits_{t\rightarrow \infty} p^*(A^{(t)}(m,n,\lambda,x_t,y_t))=p^*(A^{(0)}(m,n,\lambda))$, there exists a big enough number $T$ such that 
\begin{eqnarray}
\check{p}<p^*(A^{(t)}(m,n,\lambda,x_t,y_t))
\end{eqnarray}
for any $t\geq T$. For such $t$, there exists a vector $\hat{h}(t) \in N(A^{(t)}(m,n,\lambda,x_t,y_t))$ such that $\hat{h}^T(t)=(\bar{h}^T,\hat{h}(t)_1,\hat{h}(t)_2,\ldots,\hat{h}(t)_{m+2})$, where $\hat{h}(t)_i\in \mathbb R$ ($i=1,2,\ldots ,m+2$).

By the definition of $\{x_t\}$ (\ref{xt}) and $\{y_t\}$ (\ref{yt}), it is easy to get that
\begin{eqnarray}
\hat{h}(t)_1=-(m+1)^{\frac{1}{\check{p}}}\cdot\frac{1}{t}
\end{eqnarray}
\begin{eqnarray}
\hat{h}(t)_2=-\frac{1}{t}
\end{eqnarray}
and
\begin{eqnarray}
|\hat{h}(t)_i|\leq \frac{1}{t}
\end{eqnarray}
for $i=3,4,\ldots ,m+2$. Therefore, we can get that
\begin{eqnarray}
|\hat{h}(t)_1|^{\check{p}}=(m+1)\frac{1}{t^{\check{p}}}\geq \sum_{i=2}^{m+2}|\hat{h}(t)_i|^{\check{p}}
\end{eqnarray}

Now, we consider a vector $\check{x}(t)^T=(x^{*T}, (m+1)^{\frac{1}{\check{p}}}\cdot\frac{1}{t},0,\ldots , 0)\in \mathbb R ^{n+m+2}$, and it is easy to get that $\|\check{x}(t)\|_0=\|x^*\|_0+1\leq m+1$.

Since $x^*$ is the solution of $l_0$-minimization, it is obvious that $\check{x}(t)$ is also the solution of the responding $l_0$-minimization with a higher dimension. By Theorem 1 and Proposition 1, we can get that $\check{p}<p^*(A^{(t)}(m,n,\lambda))$ for a big enough $t$. Since $spark(A^{(t)}(m,n,\lambda,x_t,y_t))=2m+3$, we have that
\begin{eqnarray}
\notag \|\check{x}(t)\|_{\check{p}}^{\check{p}}=\|x^*\|_{\check{p}}^{\check{p}}+(m+1)\frac{1}{t^{\check{p}}}&<&\|\check{x}(t)+\hat{h}(t)\|_{\check{p}}^{\check{p}},\\
&=&\|x^*+\bar{h}\|_{\check{p}}^{\check{p}}+\sum_{i=2}^{m+2}|\hat{h}(t)_i|^{\check{p}}.
\end{eqnarray}

It is obvious that $\|x^*\|_{\check{p}}^{\check{p}}<\|x^*+\bar{h}\|_{\check{p}}^{\check{p}}$.
Therefore, the proof is completed.
\end{proof}
\begin{remark}
The conclusion $\lim \limits_{t\rightarrow \infty} \lambda_{min^+}(A^{(t)}(m,n,\lambda,x_t,y_t))=\lambda_{min^+}(A^{(0)}(m,n,\lambda))$ is of greatest importance while RIC and RIP do not always satisfy this limit property.
\end{remark}
\begin{remark}
In order to construct such vector $\lambda^*\in \mathbb R^{2m+2}$, the absolute value of the constant $\lambda_i$ $(i\in \{m+1,m+2,...,2m+2\})$ just need to be different from that of $\lambda_j$ $(j\in\{1,2,...,m\})$. So there exist a lot of vectors meeting the this condition. It is worthy studying how to find a reasonable vector to get the optimal $p^*(A^{(0)}(m,n,\lambda^*))$.
\end{remark}
\section{Conclusion}
In this paper, we present an analytic expression of $p$ for a given Vandermonde matrix $A(m,n,\lambda)\in \mathbb R^{m \times n}$ such that the solutions of $l_0$-minimization is the solution of $l_p$-minimization. Different to related work based on RIP, we fundamentally give a answer to this equivalence problem between $l_0$-minimization and $l_p$-minimization, and the solution of $l_0$-minimization $x^*$ do not need to be assumed to be unique solution and $\|x\|_0<\frac{spark(A)}{2}$.

As we have already mentioned, RIP and RIC only consider the case when $l_0$-minimization only has an unique solution and these two concepts require $2\|x\|_0<spark(A)$. Different to the results based on RIP and RIC, we do not need the uniqueness assumption and we consider a more general case including the case where $\|x\|_0\geq\frac{1}{2}spark(A)$.

It should be pointed out that Theorem 1 is also can be used in any underdetermind matrix $A$ with $\|x^*\|<\frac{spark(A)}{2}$, including the matrices with RIP of $2k$ order and $\|x^*\|_0\leq k$. The advantage of our result its computability, i.e., each part in this analytic expression can be easily calculated. As we know, to calculate RIC for a given matrix which is satisfied with RIP is also NP-hard. The authors think that the method used in Section 3 can also be used in other types of matrices. In conclusion, the authors hope that in publishing this paper, a brick will be thrown out and be replaced with a gem.
\bibliographystyle{plain}
\bibliography{x}

\begin{thebibliography}{10}

\bibitem{candes2008restricted}
Emmanuel~J Cand{\`e}s.
\newblock The restricted isometry property and its implications for compressed
  sensing.
\newblock {\em Comptes Rendus Mathematique}, 346(9):589--592, 2008.

\bibitem{candes2009exact}
Emmanuel~J Cand{\`e}s and Benjamin Recht.
\newblock Exact matrix completion via convex optimization.
\newblock {\em Foundations of Computational mathematics}, 9(6):717--772, 2009.

\bibitem{candes2005decoding}
Emmanuel~J Candes and Terence Tao.
\newblock Decoding by linear programming.
\newblock {\em IEEE Transactions on Information Theory,}, 51(12):4203--4215,
  2005.

\bibitem{candes2006near}
Emmanuel~J Candes and Terence Tao.
\newblock Near-optimal signal recovery from random projections: Universal
  encoding strategies?
\newblock {\em IEEE Transactions on Information Theory,}, 52(12):5406--5425,
  2006.

\bibitem{chartrand2007exact}
Rick Chartrand.
\newblock Exact reconstruction of sparse signals via nonconvex minimization.
\newblock {\em Signal Processing Letters, IEEE}, 14(10):707--710, 2007.

\bibitem{donoho2005sparse}
David~L Donoho and Jared Tanner.
\newblock Sparse nonnegative solution of underdetermined linear equations by
  linear programming.
\newblock {\em Proceedings of the National Academy of Sciences of the United
  States of America}, 102(27):9446--9451, 2005.

\bibitem{Foucart2012Sparse}
Simon Foucart.
\newblock Sparse recovery algorithms: Sufficient conditions in terms of
  restrictedisometry constants.
\newblock {\em Springer Proceedings in Mathematics}, 13:65--77, 2012.

\bibitem{foucart2009sparsest}
Simon Foucart and Ming-Jun Lai.
\newblock Sparsest solutions of underdetermined linear systems via
  q-minimization for $0<q<1$.
\newblock {\em Applied and Computational Harmonic Analysis}, 26(3):395--407,
  2009.

\bibitem{foucart2013mathematical}
Simon Foucart and Holger Rauhut.
\newblock {\em A mathematical introduction to compressive sensing}, volume~1.
\newblock Springer, 2013.

\bibitem{Fuchs2005Sparsity}
J.~J Fuchs.
\newblock Sparsity and uniqueness for some specific under-determined linear
  systems.
\newblock In {\em IEEE International Conference on Acoustics, Speech, and
  Signal Processing, 2005. Proceedings}, pages v/729--v/732 Vol. 5, 2005.

\bibitem{gribonval2003sparse}
R{\'e}mi Gribonval and Morten Nielsen.
\newblock Sparse representations in unions of bases.
\newblock {\em IEEE Transactions on Information Theory}, 49(12):3320--3325,
  2003.

\bibitem{Liu2019A}
Qiuchen Liu and Yong Wang.
\newblock A fast eigenvector-based autofocus method for sparse aperture isar
  sensors imaging of moving target.
\newblock {\em IEEE Sensors Journal}, 19(4):1307--1319, 2019.

\bibitem{malioutov2005sparse}
Dmitry Malioutov, M{\"u}jdat {\c{C}}etin, and Alan~S Willsky.
\newblock A sparse signal reconstruction perspective for source localization
  with sensor arrays.
\newblock {\em IEEE Transactions on Signal Processing,}, 53(8):3010--3022,
  2005.

\bibitem{Marques2016Sampling}
Antonio~G. Marques, Santiago Segarra, Geert Leus, and Alejandro Ribeiro.
\newblock Sampling of graph signals with successive local aggregations.
\newblock {\em IEEE Transactions on Signal Processing}, 64(7):1832--1843, 2016.

\bibitem{natarajan1995sparse}
Balas~Kausik Natarajan.
\newblock Sparse approximate solutions to linear systems.
\newblock {\em SIAM journal on computing}, 24(2):227--234, 1995.

\bibitem{olshausen1996emergence}
Bruno~A Olshausen et~al.
\newblock Emergence of simple-cell receptive field properties by learning a
  sparse code for natural images.
\newblock {\em Nature}, 381(6583):607--609, 1996.

\bibitem{peng2015np}
Jigen Peng, Shigang Yue, Haiyang Li, et~al.
\newblock Np/cmp equivalence: a phenomenon hidden among sparsity models l\_
  $\{$0$\}$ minimization and l\_ $\{$p$\}$ minimization for information
  processing.
\newblock {\em IEEE Transactions on Information Theory}, 61(7):4028--4033,
  2015.

\bibitem{petukhov2006fast}
Alexander Petukhov.
\newblock Fast implementation of orthogonal greedy algorithm for tight wavelet
  frames.
\newblock {\em Signal processing}, 86(3):471--479, 2006.

\bibitem{Qiao2017Gridless}
Heng Qiao and Piya Pal.
\newblock Gridless line spectrum estimation and low-rank toeplitz matrix
  compression using structured samplers: A regularization-free approach.
\newblock {\em IEEE Transactions on Signal Processing}, PP(99):1--1, 2017.

\bibitem{tropp2004greed}
Joel~A Tropp.
\newblock Greed is good: Algorithmic results for sparse approximation.
\newblock {\em IEEE Transactions on Information Theory,}, 50(10):2231--2242,
  2004.

\bibitem{wright2009robust}
John Wright, Allen~Y Yang, Arvind Ganesh, Shankar~S Sastry, and Yi~Ma.
\newblock Robust face recognition via sparse representation.
\newblock {\em Pattern Analysis and Machine Intelligence, IEEE Transactions
  on}, 31(2):210--227, 2009.

\bibitem{Wang2018Super}
Wang Yong and Liu Qiuchen.
\newblock Super-resolution sparse aperture isar imaging of maneuvering target
  via the relax algorithm.
\newblock {\em IEEE Sensors Journal}, pages 1--1, 2018.

\end{thebibliography}
\end{document}